\documentclass[12pt,a4paper,english,reqno]{amsart}
\usepackage[a4paper,footskip=1.5em]{geometry}
\usepackage{amsmath,amssymb,amsthm,mathtools,mathrsfs}
\usepackage{adjustbox,tikz,calc,graphics,babel}
\usepackage[final]{microtype}
\usepackage{csquotes,enumerate,verbatim}
\usepackage[numbers]{natbib}
\usetikzlibrary{shapes.misc,calc,intersections,patterns,decorations.pathreplacing}
\usepackage{hyperref}
\hypersetup{colorlinks=true,linkcolor=blue,citecolor=blue}

\theoremstyle{plain}
\newtheorem*{theorem*}{Theorem}
\newtheorem{theorem}{Theorem}[section]
\newtheorem{lemma}[theorem]{Lemma}

\newtheorem*{claim*}{Claim}
\newtheorem{corollary}[theorem]{Corollary}
\newtheorem{conjecture}[theorem]{Conjecture}
\newtheorem{problem}[theorem]{Problem}
\theoremstyle{remark}

\def\N{\mathbb{N}}

\def\C{\mathcal}

\let\emptyset\varnothing

\let\originalleft\left
\let\originalright\right
\renewcommand{\left}{\mathopen{}\mathclose\bgroup\originalleft}
\renewcommand{\right}{\aftergroup\egroup\originalright}

\makeatletter
\def\imod#1{\allowbreak\mkern10mu({\operator@font mod}\,\,#1)}
\makeatother

\begin{document}

\title{Exactly $m$-coloured complete infinite subgraphs}

\author{Bhargav Narayanan}
\address{Department of Pure Mathematics and Mathematical Statistics, University of Cambridge, Wilberforce Road, Cambridge CB3\thinspace0WB, UK}
\email{b.p.narayanan@dpmms.cam.ac.uk}

\date{22 February 2013}

\begin{abstract}
Given an edge colouring of a graph with a set of $m$ colours, we say that the graph is (\emph{exactly}) \emph{$m$-coloured} if each of the colours is used. The question of finding exactly $m$-coloured complete subgraphs was first considered by Erickson in 1994; in 1999, Stacey and Weidl partially settled a conjecture made by Erickson and raised some further questions. In this paper, we shall study, for a colouring of the edges of the complete graph on $\N$ with exactly $k$ colours, how small the set of natural numbers $m$ for which there exists an $m$-coloured complete infinite subgraph can be. We prove that this set must have size at least $\sqrt{2k}$; this bound is tight for infinitely many values of $k$. We also obtain a version of this result for colourings that use infinitely many colours.
\end{abstract}

\maketitle

\section{Introduction}
A classical result of Ramsey~\citep{Ramsey1930} says that when the edges of a complete graph on a countably infinite vertex set are finitely coloured, one can always find a complete infinite subgraph all of whose edges have the same colour. 

Ramsey's Theorem has since been generalised in many ways; most of these generalisations are concerned with finding other monochromatic structures. For a survey of many of these generalisations, see the book of Graham, Rothschild and Spencer~\citep{Graham1990}. Ramsey theory has witnessed many developments over the last fifty years and continues to be an area of active research today; see~\citep{Leader2012, Hindman2008, Thomason1988, Conlon2009}, for instance. 

Alternatively, anti-Ramsey theory, which originates in a paper of Erd{\H{o}}s,  Simonovits and S{\'o}s~\cite{Erdos1975}, is concerned with finding large `rainbow coloured' or `totally multicoloured' structures. Between these two ends of the spectrum, one could consider the question of finding structures which are coloured with exactly $m$ different colours as was first done by Erickson~\citep{Erickson1994}; it is this line of enquiry that we pursue here.

\section{Our results}
For a set $X$, denote by $X^{(2)}$ the set of all unordered pairs of elements of $X$; equivalently, $X^{(2)}$ is the complete graph on the vertex set $X$. As usual, we write $[n]$ for $\{1,\dots,n\}$, the set of the first $n$ natural numbers. We denote a surjective map $f$ from a set $X$ to another set $Y$ by $f: X \twoheadrightarrow Y$. By a \emph{colouring} of a graph, we mean a colouring of the edges of the graph unless we specify otherwise.

Let $\Delta:\N^{(2)}\twoheadrightarrow \C{C}$ be a surjective colouring of the edges of the complete graph on $\N$ with an arbitrary set of colours $\C{C}$. If the set of colours $\C{C}$ is infinite, we say that $\Delta$ is an \emph{infinite-colouring} and if $\C{C}$ is finite, we say that $\Delta$ is a \emph{$k$-colouring} if $|\C{C}|=k$.

Given a colouring $\Delta:\N^{(2)}\twoheadrightarrow \C{C}$ of the complete graph on $\N$, we say that a subset $X$ of $\N$ is (\emph{exactly}) $m$-\emph{coloured} if $\Delta(X^{(2)})$, the set of values attained by $\Delta$ on the edges with both endpoints in $X$, has size exactly $m$. Let $\gamma_{\Delta}(X)$, or $\gamma(X)$ in short, denote the size of the set $\Delta( X^{(2)} )$; in other words, every set $X$ is $\gamma(X)$-coloured. Our aim in this paper is to study the set 
\[
\C{F}_{\Delta} = \left\{ \gamma_\Delta(X)  :  X\subset\N \mbox{ such that } X \mbox{ is infinite}\right\} .
\]

We first consider colourings using finitely many colours. Let $\Delta:\N^{(2)}\twoheadrightarrow [k]$ be a $k$-colouring of the edges of the complete graph on the natural numbers with $k \ge 2$ colours. Trivially, $k\in\C{F}_{\Delta}$ since $\Delta$ is surjective, and Ramsey's Theorem tells us that $1\in\C{F}_{\Delta}$.  Furthermore, as was noted by Erickson~\citep{Erickson1994}, a fairly straightforward application of Ramsey's Theorem enables one to show that $2\in\C{F}_{\Delta}$ for any $k$-colouring $\Delta$ with $k\ge 2$. Erickson conjectured however that with the exception of $1$, $2$ and $k$, no other elements are guaranteed to be in $\C{F}_{\Delta}$. 

\begin{conjecture}
\label{mcol12-erickson}
Let $k,m \in \N$ with $k>m>2$. Then there exists a $k$-colouring $\Delta:\N^{(2)}\twoheadrightarrow [k]$ such that $m\notin\C{F}_{\Delta}$.
\end{conjecture}

Stacey and Weidl~\citep{Stacey1999} settled this conjecture in the case where $k$ is much bigger than $m$. More precisely, for any $m>2$, they showed that there exists a constant $C_{m}$ such that if $k>C_{m}$, then there is a $k$-colouring $\Delta$ such that $m\notin\C{F}_{\Delta}$.

Erickson's conjecture, if true, would suggest that it is hopeless to look for particular values in the set $\C{F}_{\Delta}$ given a $k$-colouring $\Delta:\N^{(2)}\twoheadrightarrow [k]$. It is natural then to consider other properties of the set $\C{F}_{\Delta}$. The first question which arises is that of the set of possible sizes of $\C{F}_{\Delta}$. Since $\C{F}_{\Delta}\subset[k]$, it follows that $|\C{F}_{\Delta}|\le k$ and it is easy to see that equality is in fact possible. Things are not so clear when we turn to the question of lower bounds. Let us define 
\[
\psi(k) = \min_{\Delta:\N^{(2)}\twoheadrightarrow[k]}|\C{F}_{\Delta}|.	
\]
We are able to prove the following lower bound for $\psi(k)$.

\begin{theorem}
\label{mcol12-mainresult}
Let $n \ge 2$ be the largest natural number such that $k\ge\binom{n}{2}+1$. Then $\psi(k)\ge n$.
\end{theorem}

It is not hard to check that Theorem~\ref{mcol12-mainresult} is tight when $k=\binom{n}{2}+1$ for some $n\ge2$. To this end, we consider the `small-rainbow colouring' $\Delta$ which colours all the edges with both endpoints in $[n]$ with $\binom{n}{2}$ distinct colours and all the remaining edges with the one colour that has not been used so far. Clearly, $\C{F}_{\Delta}= \{ \binom{i}{2}+1 : i \le n\}$, so Theorem~\ref{mcol12-mainresult} is best-possible for infinitely many values of $k$.

Turning to the question of upper bounds for $\psi$, the small-rainbow colouring demonstrates that $\psi (k) = O(\sqrt{k})$ for infinitely many values of $k$. When $k$ is not of the form $\binom{n}{2}+1$, there are two obvious ways of generalising the small-rainbow colouring described above: we could replace the rainbow coloured clique in the construction either with a disjoint union of cliques, or with a clique along with an extra vertex attached to some vertices of the clique. It is not hard to check that both these generalisations fail to give us good upper bounds for $\psi (k)$ for general $k$; in particular, we are unable to decide if $\psi(k)=o(k)$ for all $k \in \N$. However, by considering colourings that colour all the edges of a small complete bipartite graph with distinct colours (as opposed to a small complete graph) and making use of some number theoretic estimates of Tenenbaum~\citep{Tenenbaum1984} and Ford~\citep{Ford2008}, we get reasonably close to such a statement.

\begin{theorem}
\label{mcol12-upperbound}
There exists a subset $A$ of the natural numbers of asymptotic density one such that for all $k\in A$, 
\[\psi(k) = O \left(\frac{k}{(\log{\log{k}})^{\delta}(\log{\log{\log{k}}})^{3/2}}\right),\]  
where $\delta = 1 - \frac{1 + \log{\log{2}} }{\log{2}} \approx 0.086 > 0$.
\end{theorem}

In the spirit of canonical Ramsey theory, which originates in a paper of Erd\H{o}s and Rado~\citep{Erdos1950}, we also study colourings using infinitely many colours. When $\Delta$ is an infinite-colouring, then it might so happen (when $\Delta$ is injective, for instance) that for each infinite subset $X$ of $\N$, the set $\Delta(X^{(2)})$ is infinite; consequently, our search for infinite $m$-coloured subsets is doomed to fail in this case. So given a colouring $\Delta:\N^{(2)}\twoheadrightarrow \C{C}$, we define
\[
\C{G}_{\Delta}=\left\{ \gamma_\Delta(X)  :  X\subset \N \right\} .
\]
The difference between $\C{G}_{\Delta}$ and $\C{F}_{\Delta}$ is that we also consider finite complete subgraphs in defining $\C{G}_{\Delta}$. We can prove the following analogue of Theorem~\ref{mcol12-mainresult} for infinite-colourings.

\begin{theorem}
\label{mcol12-canonical}
Let $\Delta:\N^{(2)}\twoheadrightarrow\N$ be an infinite-colouring and suppose $n\ge2$ is a natural number. Then $|\C{G}_{\Delta}\cap[\binom{n}{2}]|\ge n-1$.
\end{theorem}

By considering the injective colouring that colours each edge with a distinct colour, it is easy to see that Theorem~\ref{mcol12-canonical} is best-possible.

The rest of this paper is organised as follows. In the next section, we prove our lower bounds, namely Theorems~\ref{mcol12-mainresult} and~\ref{mcol12-canonical}. We remark that we do not prove Theorem~\ref{mcol12-mainresult} and~\ref{mcol12-canonical} as stated. Instead, we prove two stronger structural results that in turn imply these theorems. We postpone the statements of these results since they depend on a certain notion of homogeneity that we shall introduce in the next section. In Section~\ref{mcol12-s:upper}, we describe how Theorem~\ref{mcol12-upperbound} follows from certain divisor estimates. We conclude by mentioning some open problems in Section~\ref{mcol12-s:conc}.

\section{Lower bounds}\label{mcol12-s:low}

In this section, we prove Theorem~\ref{mcol12-mainresult} by proving a stronger structural result, namely Theorem~\ref{mcol12-nhomog}. The proof of Theorem~\ref{mcol12-canonical} via Theorem~\ref{mcol12-nwhomog} is very similar and we shall only highlight the main differences in the proofs. 

We first introduce a notational convenience. Given a colouring $\Delta$ of $\N^{(2)}$, a vertex $v\in\N$, and a subset $X\subset\N\setminus\{v\}$, we say that a colour $c$ is a \emph{new colour from} $v$ \emph{into} $X$ if some edge from $v$ to $X$ is coloured $c$ by $\Delta$ and also, no edge of $X^{(2)}$ is coloured $c$ by $\Delta$. We write $N_{\Delta}(v,X)$, or just $N(v,X)$ when the colouring $\Delta$ in question is clear, for the set of new colours from $v$ into $X$.

\subsection{Proof of Theorem~\ref{mcol12-mainresult}}

Before we prove Theorem~\ref{mcol12-mainresult}, we note that Erickson's argument showing that $2\in\C{F}_{\Delta}$ can be generalised to give a quick proof of the fact that $\psi(k)=\Omega(\log k)$.

\begin{lemma}
\label{mcol12-erickson-lem}
Let $\Delta:\N^{(2)}\twoheadrightarrow[k]$ be a $k$-colouring and suppose $l\in\C{F}_{\Delta}$ and $l<k$. Then there is an $m\in\C{F}_{\Delta}$ such that $l+1\le m\le2l$.
\end{lemma}

Note that Lemma~\ref{mcol12-erickson-lem}, coupled with the fact that we always have $1\in\C{F}_{\Delta}$, implies that $\psi(k)\ge1+\log_{2}k$.

\begin{proof}[Proof of Lemma~\ref{mcol12-erickson-lem}]
Let $X\subset\N$ be a maximal $l$-coloured set. As $l<k$, $X\neq\N$. Pick $v\in\N\setminus X$. Note that $N(v,X)\neq\emptyset$ since otherwise $X\cup\{v\}$ is $l$-coloured, which contradicts the maximality of $X$. 

If $|N(v,X)|\le l$, then $X\cup\{v\}$ is $m$-coloured for some $l+1\le m\le2l$. So suppose $|N(v,X)|\ge l+1$. By the pigeonhole principle, there is an infinite subset $Y$ of $X$ such that all the vertices of $Y$ are connected to $v$ by edges of a single colour, say $c$.

We consider two cases. If $c\in N(v,X)$, we pick $l-1$ vertices from $X$ which are joined to $v$ by edges coloured with $l-1$ distinct colours from $N(v,X)\setminus\{c\}$. If on the other hand $c\notin N(v,X)$, we pick $l$ vertices from $X$ which are joined to $v$ by edges coloured with $l$ distinct colours from $N(v,X)$. Call this set of $l-1$ or $l$ vertices $Z$. 

In both cases, it is easy to check that $Y\cup Z\cup\{v\}$ is $m$-coloured with $l+1\le m\le2l$.
\end{proof}

Consequently,  we have the following corollary.

\begin{corollary}
\label{mcol12-powersof2}
If $\Delta:\N^{(2)}\twoheadrightarrow[k]$ is a $k$-colouring and $n$ is a natural number such that $k\ge2^{n}+1$, then $\C{F}_{\Delta}\cap([2^{n+1}]\setminus[2^{n}])\neq\emptyset.$ \qed
\end{corollary}

We shall show that for any $k$-colouring $\Delta:\N^{(2)}\twoheadrightarrow[k]$ with $k\ge\binom{n}{2}+1$ for some $n$, we can find $n$ nested subsets $A_1 \subsetneq A_2 \subsetneq \dots \subsetneq A_n$ of $\N$ such that $\Delta(A_{1}^{(2)})\subsetneq\Delta(A_{2}^{(2)})\subsetneq\dots\subsetneq\Delta(A_{n}^{(2)})$. To do this, we introduce the notion of $n$-homogeneity on which our first structural result, Theorem~\ref{mcol12-nhomog}, hinges. 

For an ordered $n$-tuple $\mathbf{X}=(X_{1},X_{2},\dots,X_{n})$, write $\widehat{X}_{i}$ for the set $X_{1}\cup X_{2}\dots\cup X_{i}$. Given a colouring $\Delta$, we call $\mathbf{X}=(X_{1},X_{2},\dots,X_{n})$, with each $X_{i}$ a nonempty subset of $\N$, $n$-\emph{homogeneous with respect to} $\Delta$ if the following conditions are met:

\begin{enumerate}
\item\label{mcol12-one} $X_{i}\cap X_{j}=\emptyset$ for $i\neq j$,
\item\label{mcol12-two} $X_{1}$ is infinite and $1$-coloured,
\item\label{mcol12-three} $\Delta(\widehat{X}_{1}^{(2)})\subsetneq\Delta(\widehat{X}_{2}^{(2)})\subsetneq\dots\subsetneq(\widehat{X}_{n}^{(2)})$,
\item\label{mcol12-four} for each $X_{i}$ with $2\le i\le n$, every $v\in X_{i}$ satisfies
\[
N(v,\widehat{X}_{i-1})=\Delta\left(\widehat{X}_{i}^{(2)}\right)\setminus\Delta\left(\widehat{X}_{i-1}^{(2)}\right)\mbox{, and}
\]
\item\label{mcol12-five} $\gamma(\widehat{X}_{n}) \le \binom{n}{2}+1$.
\end{enumerate}

Rather than proving Theorem~\ref{mcol12-mainresult}, we prove the following stronger statement.

\begin{theorem}
\label{mcol12-nhomog}
Let $\Delta:\N^{(2)}\twoheadrightarrow[k]$ be a $k$-colouring and suppose $n$ is a natural number such that $k\ge\binom{n}{2}+1$. Then there exists an $n$-homogeneous tuple with respect to $\Delta$.
\end{theorem}

Before we prove Theorem~\ref{mcol12-nhomog}, let us first recall the lexicographic order on $\N^r$: we say that $(a_{1},a_{2}\dots,a_{r})<(b_{1},b_{2}\dots,b_{r})$ if for some $l\le r-1$ we have $a_{i}=b_{i}$ for $1\le i\le l$ and $a_{l+1}<b_{l+1}$. 

Note that if $\mathbf{X}=(X_{1},X_{2},\dots,X_{n})$ is $n$-homogeneous, then by condition (\ref{mcol12-four}), the set $N(v,\widehat{X}_{i-1})$ is identical for all $v\in X_{i}$ for $2\le i\le n$. For $n\ge2$, define the \emph{rank} of an $n$-homogeneous tuple $\mathbf{X}$ to be the $(n-1)$-tuple $(x_{1},x_{2},\dots,x_{n-1})$, where $x_{i}$ is the number of new colours from any vertex of $X_{i+1}$ into the set $\widehat{X}_{i}$. Note that the rank of an $n$-homogeneous tuple is an $(n-1)$-tuple of natural numbers, so we can compare ranks using the lexicographic order on $\N^{n-1}$.

\begin{proof}[Proof of Theorem~\ref{mcol12-nhomog}]
We proceed by induction on $n$. The case $n=1$ is Ramsey's Theorem. Suppose that $k\ge\binom{n+1}{2}+1$ and assume inductively that at least one $n$-homogeneous tuple exists. 

From the set of all $n$-homogeneous tuples, pick one with minimal rank in the lexicographic order, say $\mathbf{X}=(X_{1},X_{2},\dots,X_{n})$. If $n=1$, the rank is immaterial; it suffices to pick $\mathbf{X}=(X_{1})$ such that $X_{1}$ is an infinite $1$-coloured set. We shall build an $(n+1)$-homogeneous tuple from $\mathbf{X}$. 

Note that $k\ge\binom{n+1}{2}+1>\binom{n}{2}+1$. Since $\Delta$ is surjective and attains at most $\binom{n}{2}+1$ different values inside $\widehat{X}_{n}$, it is clear that $\N\setminus\widehat{X}_{n}\neq\emptyset$. We consider two cases. 

\textbf{Case 1: $N(v,\widehat{X}_{n})\neq\emptyset$ for some $v\in\N\setminus\widehat{X}_{n}$.} If $|N(v,\widehat{X}_{n})|\le n$, then it is easy to check that $(X_{1},X_{2},\dots,X_{n},\{v\})$ is an $(n+1)$-homogeneous tuple and we are done. So, assume without loss of generality that $|N(v,\widehat{X}_{n})|\ge n+1$.

Let $j$ be the smallest index such that $N(v,\widehat{X}_{j})\neq\emptyset$. Since $N(v,\widehat{X}_{n})\neq\emptyset$, this minimal index $j$ exists. We now build our $(n+1)$-homogeneous tuple $\mathbf{Y}=(Y_{1},Y_{2},\dots,Y_{n+1})$ as follows. 

Set $Y_{1}=X_{1},Y_{2}=X_{2},\dots,Y_{j-1}=X_{j-1}$. We define $Y_{j}$ as follows. First, choose $c\in N(v,\widehat{X}_{j})$; note that by the minimality of $j$, $N(v,\widehat{X}_{j-1})=\emptyset$, so all the edges between $v$ and $\widehat{X}_{j}$ coloured $c$ are actually edges between $v$ and $X_{j}$. Take $Y_{j}\subset X_{j}$ to be the (nonempty) set of vertices $u\in X_{j}$ such that the edge between $v$ and $u$ is either coloured $c$ or with a colour from $\Delta(\widehat{X}_{j}^{(2)})$ (and hence a colour not in $N(v,\widehat{X}_{j})$). Note that if $j=1$, we can always choose $c$ such that $Y_{1}$ is an infinite subset of $X_{1}$.

Next, set $Y_{j+1}=\{v\}$. Now, note that the only colour from $\Delta(\widehat{Y}_{j+1}^{(2)})$ that might possibly occur in $N(v,\widehat{X}_{n})$ is $c$. So we can now choose $v_{1},v_{2},\dots,v_{n-j}$ from $X_{n}\cup X_{n-1}\dots\cup X_{j+1}\cup(X_{j}\setminus Y_{j})$ such that these $n-j$ vertices are joined to $v$ by edges which are all coloured by distinct elements of $N(v,\widehat{X}_{n})\setminus\{c\}$. Set $Y_{j+2}=\{v_{1}\},Y_{j+3}=\{v_{2}\},\dots,Y_{n+1}=\{v_{n-j}\}$. 

We claim that $\mathbf{Y}$ is an $(n+1)$-homogeneous tuple. Indeed, conditions (\ref{mcol12-one}) and (\ref{mcol12-two}) are obviously satisfied. 

To check condition (\ref{mcol12-three}), first note that $\Delta(\widehat{Y}_{1}^{(2)})\subsetneq\Delta(\widehat{Y}_{2}^{(2)})\subsetneq\dots\subsetneq\Delta(\widehat{Y}_{j-1}^{(2)})$ follows from the $n$-homogeneity of $\mathbf{X}$ since $\widehat{Y}_{i} = \widehat{X}_{i}$ for $1\le i \le j-1$. Also, $\Delta(\widehat{Y}_{j-1}^{(2)})\subsetneq\Delta(\widehat{Y}_{j}^{(2)})$ since ${Y}_{j} \subset {X}_{j}$. Next, $\Delta(\widehat{Y}_{j}^{(2)})\subsetneq\Delta(\widehat{Y}_{j+1}^{(2)})$ since $v$ is joined to at least one vertex of $Y_{j}$ by an edge coloured with $c$ and we know that $c$ is a new colour from $v$ into $\widehat{Y}_{j}$. Finally, $\Delta(\widehat{Y}_{j+1}^{(2)})\subsetneq\Delta(\widehat{Y}_{j+2}^{(2)})\subsetneq\dots\subsetneq\Delta(\widehat{Y}_{n+1}^{(2)})$ because the vertices $v_{1},v_{2},\dots,v_{n-j}$ are all joined to $v$ by edges of distinct colours and none of these colours belong to $\Delta(\widehat{X}_{n}^{(2)})$. So condition (\ref{mcol12-three}) is also satisfied. 

Condition (\ref{mcol12-four}) for each of $Y_{1},Y_{2},\dots,Y_{j}$ is equivalent to the same condition for $X_{1},X_{2},\dots,X_{j}$ respectively. Furthermore, condition (\ref{mcol12-four}) is also satisfied by each of $Y_{j+1},Y_{j+2},\dots,Y_{n+1}$ since they each contain exactly one vertex. 

Finally, we check condition (\ref{mcol12-five}). Clearly, $\Delta(\widehat{Y}_{n+1}^{(2)})$ is a subset of $\Delta(\widehat{X}_{n}^{(2)})\cup\, T$ for some subset $T$ of $N(v,\widehat{X}_{n})$ of size at most $n$. Hence, we see that $\gamma(\widehat{Y}_{n+1})\le\binom{n}{2}+1+n=\binom{n+1}{2}+1$. 

\textbf{Case 2: $N(v,\widehat{X}_{n})=\emptyset$ for every $v\in\N\setminus\widehat{X}_{n}$.} 
It is here that we use the fact that $\mathbf{X}$ has minimal lexicographic rank. To deal with this case, we will need the following lemma.

\begin{lemma}
\label{mcol12-lexlem}
Let $\mathbf{X}$ be an $n$-homogeneous tuple of minimal lexicographic rank and suppose $N(v,\widehat{X}_{n})=\emptyset$ for some $v\in\N\setminus\widehat{X}_{n}$. Then there is an $n$-homogeneous tuple $\mathbf{Y}$ such that $Y_{j}=X_{j}\cup\{v\}$ for some $j\in[n]$, and $Y_{i}=X_{i}$ for each $1\le i\le n$ with $i\neq j$.
\end{lemma}

\begin{proof}
If $N(v,\widehat{X}_{i})=\emptyset$ for each $1\le i\le n$, then $(X_{1}\cup\{v\},X_{2},\dots,X_{n})$ is $n$-homogeneous and we have $\mathbf{Y}$ as required. Hence, let $ j<n$ be the largest index such that $N(v,\widehat{X}_{j})\neq\emptyset$. So by the definition of $j$, $N(v,\widehat{X}_{i})=\emptyset$ for $j<i\le n$. We claim that $\mathbf{Y}=(X_{1},X_{2},\dots,X_{j},X_{j+1}\cup\{v\},X_{j+2},\dots,X_{n})$ is $n$-homogeneous. 

Consider a colour $c$ that belongs to $N(v,\widehat{X}_{j})$. Since $N(v,\widehat{X}_{j+1})=\emptyset$, this means that $c$ must occur in $\Delta(\widehat{X}_{j+1}^{(2)})\setminus\Delta(\widehat{X}_{j}^{(2)})$. But, by condition (\ref{mcol12-four}), for each $u\in X_{j+1}$, $N(u,\widehat{X}_{j})=\Delta(\widehat{X}_{j+1}^{(2)})\setminus\Delta(\widehat{X}_{j}^{(2)})$. Hence, $N(v,\widehat{X}_{j})\subset N(u,\widehat{X}_{j})$ for $u\in X_{j+1}$. 

Observe that since $N(v,\widehat{X}_{i})=\emptyset$ for $j<i\le n$, $N(u, \widehat{X}_{i-1})=N(u,\widehat{X}_{i-1}\cup\{v\})$ for each $u\in X_{i}$ with $j+1<i\le n$. From this, it is easy to see that $\mathbf{Y}$ is $n$-homogeneous if $N(v,\widehat{X}_{j})=N(u,\widehat{X}_{j})$ for $u\in X_{j+1}$.

So suppose that $N(v,\widehat{X}_{j})\subsetneq N(u,\widehat{X}_{j})$ for $u\in X_{j+1}$. Consider then the $n$-tuple $\mathbf{Z}=(X_{1},X_{2},\dots,X_{j},\{v\},X_{j+1},X_{j+2},\dots,X_{n-1})$. We claim that $\mathbf{Z}$ is $n$-homogeneous and has strictly smaller lexicographic rank than $\mathbf{X}$, which is a contradiction. 

We first check the $n$-homogeneity of $\mathbf{Z}$. Clearly, conditions (\ref{mcol12-one}) and (\ref{mcol12-two}) are satisfied by $\mathbf{Z}$. 

To check condition (\ref{mcol12-three}), first note that $\Delta(\widehat{Z}_{1}^{(2)})\subsetneq\Delta(\widehat{Z}_{2}^{(2)})\subsetneq\dots\subsetneq\Delta(\widehat{Z}_{j+1}^{(2)})$ follows from the $n$-homogeneity of $\mathbf{X}$ and the fact that $N(v,\widehat{X}_{j})\neq\emptyset$. Next, $\Delta(\widehat{Z}_{j+1}^{(2)})\subsetneq\Delta(\widehat{Z}_{j+2}^{(2)})$ since $N(v,\widehat{X}_{j})\subsetneq N(u,\widehat{X}_{j})$ for $u\in X_{j+1}$. Finally, we have $\Delta(\widehat{Z}_{j+2}^{(2)})\subsetneq\Delta(\widehat{Z}_{j+3}^{(2)})\subsetneq\dots\subsetneq\Delta(\widehat{Z}_{n}^{(2)})$
since we know that $N(u,\widehat{X}_{i-1}\cup\{v\})=N(u,\widehat{X}_{i-1})\neq\emptyset$ for each $u\in X_{i}$ with $j+1<i\le n$. So $\mathbf{Z}$ satisfies condition (\ref{mcol12-three}).

Condition (\ref{mcol12-four}) is satisfied trivially by each of $Z{}_{1},Z_{2},\dots,Z_{j}$. Condition (\ref{mcol12-four}) holds for $Z_{j+1}$ since $v$ is the only element in $Z_{j+1}$. We know that $N(v,\widehat{X}_{j+1})=\emptyset$. Hence, condition (\ref{mcol12-four}) holds for $Z_{j+2}$ since for any vertex $u\in Z_{j+2}=X_{j+1}$, we see that $N(u,\widehat{Z}_{j+1})=N(u,\widehat{X}_{j})\setminus N(v, \widehat{X}_{j})=\Delta(\widehat{Z}_{j+2}^{(2)})\setminus\Delta(\widehat{Z}_{j+1}^{(2)})$. Finally, condition (\ref{mcol12-four}) holds for each $Z_{i}$ with $j+2<i\le n$ by the fact that $N(u,\widehat{X}_{i-1}\cup\{v\})=N(u,\widehat{X}_{i-1})$ for each $u\in X_{i}$. 

Finally, it is easy to see that condition (\ref{mcol12-five}) holds since $N(v,\widehat{X}_{n})=\emptyset$. 

That $\mathbf{Z}$ has smaller lexicographic rank than $\mathbf{X}$ is clear from the fact that $N(v,\widehat{X}_{j})\subsetneq N(u,\widehat{X}_{j})$ for $u\in X_{j+1}$.
\end{proof}

We have assumed that $N(v,\widehat{X}_{n})=\emptyset$ for each $v\in\N\setminus\widehat{X}_{n}$. Now, $\Delta$ is surjective, so there must exist two vertices $v_{1}$ and $v_{2}$ in $\N\setminus\widehat{X}_{n}$ such that the edge joining $v_{1}$ and $v_{2}$ is coloured with a colour $c$ not in $\Delta(\widehat{X}_{n}^{(2)})$. 

Let $\mathbf{Y}$ be the $n$-homogeneous tuple that we get by applying Lemma~\ref{mcol12-lexlem} to $\mathbf{X}$ and $v_{1}$. It is then clear that $N(v_{2},\widehat{Y}_{n})=\{c\}$. Thus, $(Y_{1},Y_{2},\dots,Y_{n},\{v_{2}\})$ is an $(n+1)$-homogeneous tuple. This completes the proof of the theorem.
\end{proof}

\subsection{Proof of Theorem~\ref{mcol12-canonical}}

As we mentioned earlier, the proof of Theorem~\ref{mcol12-canonical} is very similar to that of Theorem~\ref{mcol12-mainresult} and also goes via a stronger structural result. We only highlight the main differences. 

To prove Theorem~\ref{mcol12-canonical}, we will need to alter the definition of $n$-homogeneity slightly. We shall relax condition (\ref{mcol12-two}): instead of demanding that our first set $X_{1}$ be infinite and $1$-coloured, we shall only require that $|X_{1}|=1$. 

More precisely, given a colouring $\Delta$, we call an $n$-tuple $\mathbf{X}=(X_{1},X_{2},\dots,X_{n})$, with each $X_{i}$ a nonempty subset of $\N$, \emph{weakly homogeneous with respect to} $\Delta$ if the following conditions are met:

\begin{enumerate}
\item $X_{i}\cap X_{j}=\emptyset$ for $i\neq j$,
\item $|X_{1}|=1$,
\item $\emptyset=\Delta(\widehat{X}_{1}^{(2)})\subsetneq\Delta(\widehat{X}_{2}^{(2)})\subsetneq\dots\subsetneq\Delta(\widehat{X}_{n}^{(2)})$,
\item for each $X_{i}$ with $2\le i\le n$, every $v\in X_{i}$ satisfies
\[
N(v,\widehat{X}_{i-1})=\Delta\left(\widehat{X}_{i}^{(2)}\right)\setminus\Delta\left(\widehat{X}_{i-1}^{(2)}\right)\mbox{, and}
\]
\item$\gamma(\widehat{X}_{n}) \le \binom{n}{2}$.
\end{enumerate}
Theorem~\ref{mcol12-canonical} is an easy consequence of the following stronger statement.

\begin{theorem}
\label{mcol12-nwhomog}
Let $\Delta:\N^{(2)}\twoheadrightarrow\N$ be an infinite-colouring and suppose $n\ge2$ is a natural number. Then there exists a weakly homogeneous $n$-tuple with respect to $\Delta$.\qed
\end{theorem}

The proof is essentially identical to that of Theorem~\ref{mcol12-nhomog}. Note that we only use the finiteness of the set of colours in two places in the proof of Theorem~\ref{mcol12-nhomog}. First, to produce an infinite $1$-coloured set for the base case of the induction and second, to ensure that the subset $Y_{1}$ of $X_{1}$ that we construct in the inductive step (in Case 1) is infinite. The definition of weak homogeneity gets around both these difficulties.

\section{Upper bounds}\label{mcol12-s:upper}

Erd\H{o}s proved in~\citep{Erdos1955} that for a natural number $n$, the set $P_n=\lbrace ab : a,b \le n \rbrace$ has size $o(n^2)$. We base the proof of Theorem~\ref{mcol12-upperbound} on the observation that $P_n$ is exactly the set of sizes of all induced subgraphs of a complete bipartite graph between two equal vertex classes of size $n$.

Let $H(x,y,z)$ be the number of natural numbers $n\le x$ having a divisor in the interval $(y, z]$. Tenenbaum~\citep{Tenenbaum1984} showed that 
\begin{equation}\label{mcol12-Tenenbaum}
H(x,y,z) = (1+ o(1))x \mbox{ if }\log{y}=o(\log{z}), z\le\sqrt{x}.
\end{equation}
Ford~\citep{Ford2008} proved that
\begin{equation}\label{mcol12-Ford}
H(x,y,2y) = \Theta \left( \frac{x}{(\log{y})^{\delta}(\log{\log{y}})^{3/2}}\right)  \mbox{ if } 3\le y \le \sqrt{x}, 
\end{equation}
where $\delta = 1 - \frac{1 + \log{\log{2}} }{\log{2}}$. Armed with these two facts, we can now prove Theorem~\ref{mcol12-upperbound}.

\begin{proof}[Proof of Theorem~\ref{mcol12-upperbound}]
We shall take
\[ A = \lbrace k : \exists\, a,b\in \N \mbox{ with }  k-1 = ab \mbox{ and}\, \log{k} \le a \le b\rbrace.\]
It follows from~\eqref{mcol12-Tenenbaum} that $H(x,\log{x},\sqrt{x})=(1+ o(1))x$; as an easy consequence, $A$ has asymptotic density one. Now, for a fixed $k\in A$ with $k-1 = ab$, consider a surjective $k$-colouring $\Delta$ of the complete graph on $\N$ which colours all the edges of the complete bipartite graph between $[a]$ and $[b+a]\setminus[a]$ with $ab$ distinct colours and all the other edges with the one  colour not used so far. It is easy to then see that 
\[\C{F}_\Delta = \lbrace a'b' + 1 : 1\le a' \le a, 1\le b'\le b \rbrace \cup \lbrace 1 \rbrace .\]

Now, for any element $a'b'+1 \in \C{F}_\Delta$, note that $a/2^{i+1} < a' \le a/2^i$ for some $i\ge 0$, so $a'b' \le ab/2^i$. Thus,
\[ |\C{F}_\Delta| \le 1 + \sum_{i\ge 0}{H\left(\frac{ab}{2^i},\frac{a}{2^{i+1}},\frac{a}{2^i}\right)}.\]
Using Ford's estimate~\eqref{mcol12-Ford} for $H(x,y,2y)$ and the fact that $a\ge \log{k}$, we obtain that
\[\psi(k) = O \left(\frac{k}{(\log{\log{k}})^{\delta}(\log{\log{\log{k}}})^{3/2}}\right)\]
for all $k \in A$.
\end{proof}

\section{Conclusion}\label{mcol12-s:conc}
Our results raise many questions that we cannot yet answer. We suspect that something much stronger than Corollary~\ref{mcol12-powersof2} is true.

\begin{conjecture}
\label{mcol12-triangconj}
Let $\Delta:\N^{(2)}\twoheadrightarrow[k]$ be a $k$-colouring and suppose $n\ge2$ is a natural number such that $k\ge\binom{n}{2}+2$. Then $\C{F}_{\Delta}\cap([\binom{n+1}{2}+1]\setminus[\binom{n}{2}+1])\neq\emptyset$. 
\end{conjecture}

If true, note that this statement would imply Theorem~\ref{mcol12-mainresult}. When $n=2$, the conjecture is implied by Corollary~\ref{mcol12-powersof2}. We are able to prove the first nontrivial instance of Conjecture~\ref{mcol12-triangconj}, namely that when $k\ge5$, $\C{F}_{\Delta}\cap\{5,6,7\}\neq\emptyset$, but the proof we possess sheds no light on how to prove the conjecture in general.

We strongly suspect that the function $\psi$ is quite far from being monotone. We have shown that $\psi(\binom{n}{2}+1)=n$ and $\psi(\binom{n+1}{2}+1)=n+1$, and it is an easy consequence of our results that $\psi(\binom{n}{2}+2)=n+1$. It appears to be true that even $\psi(\binom{n}{2}+3)$ is much bigger than $n$, though we cannot even prove much more than the fact that $\psi(\binom{n}{2}+3)>n+1$.

\begin{conjecture}
There is an absolute constant $c>0$ such that $\psi(\binom{n}{2}+3)>(1+c)n$ for all natural numbers $n\ge2$.
\end{conjecture}

The problem of determining $\psi$ completely is of course still open. We do not know the answer to even the following question.

\begin{problem}
Is $\psi(k)=o(k)$ for all $k\in\N$? 
\end{problem}

If we restrict our attention to colourings which use every colour but one exactly once, we are led to the following question about induced subgraphs, a positive answer to which would immediately imply that $\psi(k)=o(k)$ for all $k\in\N$. To state the question, we need some definitions: let $S(G)$ denote the set of sizes of all the induced subgraphs of a graph $G$ and let $S(m)$ be the minimum value of $|S(G)|$ taken over all graphs $G$ with $m$ edges.
\begin{problem}
Is $S(m)=o(m)$? 
\end{problem}

\section*{Acknowledgements} I would like to thank my supervisor B\'{e}la Bollob\'{a}s for bringing the question considered in this paper to my attention and for his many helpful comments about the presentation.

\bibliographystyle{amsplain}
\bibliography{m_col}

\end{document}